\documentclass[12pt]{article}

\usepackage{amsmath,amssymb,amscd,amsthm,makeidx,txfonts,graphicx,url,bm}
\usepackage{a4wide,pstricks,xypic}
\usepackage{eucal} 
\usepackage{xcolor}
\usepackage{youngtab}
\usepackage{young}

\numberwithin{equation}{subsection}

\let\oldmarginpar\marginpar
\renewcommand\marginpar[1]{\-\oldmarginpar[\raggedleft\footnotesize #1]
{\raggedright\footnotesize #1}}

\makeindex

\xyoption{all}
\input{xypic}

\newtheorem{theorem}{Theorem}[subsection]
\newtheorem{proposition}[theorem]{Proposition}
\newtheorem{corollary}[theorem]{Corollary}
\newtheorem{conjecture}[theorem]{Conjecture}
\newtheorem{lemma}[theorem]{Lemma}

\theoremstyle{remark}
\newtheorem{remark}[theorem]{Remark}

\theoremstyle{definition}

\def\beq{\begin{eqnarray}}
\def\eeq{\end{eqnarray}}
\def\bes{\begin{eqnarray*}}
\def\ees{\end{eqnarray*}}

\usepackage{youngtab}
\usepackage{young}

\newcommand{\C}{\mathbb{C}}
\newcommand{\Q}{\mathbb{Q}}
\newcommand{\Z}{\mathbb{Z}}

\newcommand{\Gp}{\Gamma_+}
\newcommand{\Gm}{\Gamma_-}

\newcommand{\Ht}{\tilde{H}}

\DeclareMathOperator{\ldeg}{ldeg}

\DeclareMathOperator{\Hilb}{Hilb}
\DeclareMathOperator{\Ext}{Ext}

\DeclareMathOperator{\ch}{ch}

\DeclareMathOperator{\height}{ht}
\DeclareMathOperator{\sgn}{sgn}

\DeclareMathOperator{\Tr}{Tr}

\DeclareMathOperator{\Res}{Res}

\def\C{\mathbb{C}}

\newcommand{\conjeq}{\stackrel{?}{=}}

\def\H{\mathbb{H}}

\def\F{\mathbb{F}}
\def\Q{\mathbb{Q}}

\def\Z{\mathbb{Z}}

\newcommand{\nc}{\newcommand}

\nc{\op}[1]{\mathop{\mathchoice{\mbox{\rm #1}}{\mbox{\rm #1}}
{\mbox{\rm \scriptsize #1}}{\mbox{\rm \tiny #1}}}\nolimits}
\nc{\al}{\alpha}

\nc{\ep}{\varepsilon} \nc{\ga}{\gamma} \nc{\Ga}{\Gamma}
\nc{\la}{\lambda} \nc{\La}{\Lambda} \nc{\si}{\sigma}
\nc{\Sig}{{\Gamma}} \nc{\Om}{\Omega} \nc{\om}{\omega}

\nc{\SL}{{\rm SL}}
\nc{\GL}{{\rm GL}}
\nc{\PGL}{{\rm PGL}}
\nc{\G}{{\rm G}}
\nc{\Sp}{{\rm Sp}}

\nc{\Frob}{\op{ Frob}}
\nc{\Irr}{\op{Irr}}
\nc{\End}{\op{End}}

\nc{\cpt}{{\op{cpt}}} \nc{\Dol}{{\op{Dol}}} \nc{\DR}{{\op{DR}}}
\nc{\B}{{\op{B}}} \nc{\Triv}{\op{Triv}} \nc{\Hod}{{\op{Hod}}}
\nc{\Log}{{\op{Log}}} \nc{\Exp}{{\op{Exp}}} \nc{\Est}{E_{\op{st}}}
\nc{\Hst}{H_{\op{st}}} \nc{\Left}[1]{\hbox{$\left#1\vbox to
   10.5pt{}\right.\nulldelimiterspace=0pt \mathsurround=0pt$}}
\nc{\Right}[1]{\hbox{$\left.\vbox to
   10.5pt{}\right#1\nulldelimiterspace=0pt \mathsurround=0pt$}}
\nc{\LEFT}[1]{\hbox{$\left#1\vbox to
   15.5pt{}\right.\nulldelimiterspace=0pt \mathsurround=0pt$}}
\nc{\RIGHT}[1]{\hbox{$\left.\vbox to
   15.5pt{}\right#1\nulldelimiterspace=0pt \mathsurround=0pt$}}

\nc{\bee}{{\bf E}} \nc{\bphi}{{\bf \Phi}}

\begin{document}

\title{Vertex operators and character varieties}
\author{Erik Carlsson\\
\and Fernando Rodriguez Villegas\\
\\ {\it CMSA, Harvard University,}
{\tt carlsson@cmsa.fas.harvard.edu}
\\ {\it ICTP,}
{\tt villegas@ictp.it}}
\maketitle


\abstract{We prove some combinatorial conjectures extending those
  proposed in~\cite{hausel-villegas,HaLeVi1}. The proof uses a
  \emph{vertex operator} due to Nekrasov, Okounkov, and the first
  author~\cite{CaNekOk1} to obtain a ``gluing formula'' for the
  relevant generating series, essentially reducing the computation to
  the case of $\mathbb{CP}^1$ with three punctures.}

\section{Introduction}

Let $g\geq 0$ be a non-negative integer. For a partition 
$\lambda\in \mathcal{P}$ let 
$$
\mathcal{H}_\lambda(u;z,w):=
\prod_{s\in 
  \lambda}\frac{\prod_{i=1}^g(z^{2a(s)+1}-u_iw^{2l(s)+1}) 
  (z^{2a(s)+1}-u_i^{-1}w^{2l(s)+1})}  
{(z^{2a(s)+2}-w^{2l(s)})(z^{2a(s)}-w^{2l(s)+2})},
$$
where the product is over all cells $s$ of $\lambda$ with $a(s)$ and
$l(s)$ its arm and leg length, respectively and
$z,w,u=(u_1,\ldots,u_g)$ are independent variables. This is a
generalization of the hook function of~\cite{hausel-villegas} which we
recover by setting $u_i=1$ for $i=1,\ldots, g$. When $g=1$ there is
only one variable $u_i$, which we will just denote by $u$.

Given nonnegative integers $g,k$, let us define
\begin{equation}
\label{Omega}
\Omega(u;z,w):=
\sum_{\lambda} 
\mathcal{H}_\lambda(u;z,w)\, \tilde{H}_\lambda[X_1]\cdots
\tilde{H}_\lambda[X_k]
\end{equation}
where $|\lambda|$ is the size of the partition $\lambda$, and
$\tilde{H}_\lambda[X_i]$ is the modified Macdonald polynomial defined
below, in an infinite set of variables
$X_i=\left\{x_{i1},x_{i2},\ldots\right\}$, with the usual parameters
specialized to $(q,t)=(z^2,w^2)$.  If $k=0$ we will include in the
definition of $\Omega$ a factor of $T^{|\lambda|}$, where $T$ is a
formal parameter, to keep track of the degree.
Consider the power series
\[
\H(u;z,w):=
(z^2-1)(1-w^2) \,
\Log \,\Omega(u;z,w),
\]
where $\Log$ is the \emph{plethystic logarithm}, defined below.
If necessary for clarification we will add the subscript $g$ or $g,k$ in
this order.

Let us define coefficients
\begin{equation}
\label{Hzw}
\H(u;z,w)=:
 \sum_{\bm{\lambda}}
\mathbb{H}_{\bm{\lambda}}(u;z,w)\, m_{\lambda^1}[X_1]\cdots
m_{\lambda^k}[X_k] 
\end{equation}
where the sum is over all
$\bm{\lambda}=(\lambda^1,\ldots,\lambda^k)\in \mathcal{P}^k$. Note that
$\mathbb{H}_{\bm{\lambda}}(u;z,w)=0$ unless $|\lambda^i|=n$ for every
$i$. 

For $k=0$ there are no symmetric functions and the polynomials are
simply indexed by their degree $n$; we will denote them by
$\H_{(n)}(u;z,w)$.
Note that this notation is consistent, in the sense that the degree $n$
term for $k=0$ is indeed the coefficient of $m_{(n)}(X)$ for
$k=1$. To see this it is enough to note that if we set
$X=(T,0,\cdots)$ then $\tilde H_\lambda(X)=T^{|\lambda|}$ for all
$\lambda$ and
$m_{\lambda}(X)$ vanishes unless $\lambda=(n)$ for some $n$ when it
equals $T^n$.

The specializations
$\mathbb{H}_{\bm{\lambda}}(z,w):=
\mathbb{H}_{\bm{\lambda}}(1,\ldots,1;z,w)$ are precisely the
coefficients of~\cite{HaLeVi1} where the following conjecture was put
forward. 
\begin{conjecture}(~\cite{HaLeVi1}[Conj. 1.2.1])
\label{genconj}
\begin{enumerate}
\item The rational function $\mathbb{H}_{\bm{\lambda}}(z,w)$
is in fact a polynomial, and $\mathbb{H}_{\bm{\lambda}}(-z,w)$
has nonnegative integer coefficients.
\item Moreover, the mixed Hodge polynomial of the character variety
  $\mathcal{M}_{\bm{\lambda}}$ of a Riemann surface of genus $g$ with
  $k$ punctures and generic semi-simple conjugacy classes of type
  $\bm{\lambda}$ is given by
\[H_c(\mathcal{M}_{\bm{\lambda}};q,t)= (t\sqrt{q})^{d_{\bm{\mu}}}
\mathbb{H}_{\bm{\lambda}}\left(-\frac{1}{\sqrt{q}},t\sqrt{q}\right).\]
\end{enumerate}
\end{conjecture}

In certain special cases in which the mixed Hodge polynomials are
known, Conjecture~\ref{genconj} reduces to a purely combinatorial
identity.  For instance, in~\cite{hausel-villegas}[Conj. 1.1.2], the
authorsconjectured an explicit formula in the special case of genus
one, and no punctures:
\begin{equation}
\label{conj}
\H_1(z,w)\conjeq
(z-w)^2\frac{T}{1-T}.
\end{equation}
Notice the positivity after substituting $z=-z$.  Similarly,
combining~\eqref{conj} with Conjecture~4.2 of~\cite{HaLeVi2}, which
relies on known facts on the geometry of the Hilbert scheme of points
on $\C^*\times\C^*$, we expect the following more involved identity to
hold:
\begin{equation}
\label{tenniseq}
\H_{1,\psi}(z,w)=\H_{1,\psi}(1;z,w)\conjeq(z-w)^2(z^2+w^2-z^2w^2)\frac{T}{1-T}, 
\end{equation}
where
\[
\H_{1,\psi}(u;z,w):=
(z^2-1)(1-w^2)\ \Log \sum_{\lambda} \mathcal{H}_{1,\lambda}(u;z,w) 
\psi_\lambda(z^2,w^2) T^{|\lambda|}
\]
and
\[
\psi_{\lambda}(q,t):=1+(1-q)(1-t) B_\lambda(q,t),\qquad
B_\lambda(q,t):=\sum_{(i,j)\in \lambda} q^{j-1} t^{i-1}.
\]
We note that the $\psi_\lambda(q,t)$'s are well known to be the eigenvalues
of the modified Macdonald operator $D_0$ of~\cite{Ber} described
below.

In this paper we prove the following more general versions of the
identities~\eqref{conj} and~\eqref{tenniseq}
(see~\S\ref{proof-main-ident} for the proof).
\begin{theorem}
\label{main-ident}
The following identities for $g=1$ hold
\begin{equation}
\label{conju}
\H_1(u;z,w)=(z-u^{-1}w)(z-uw)\frac{T}{1-T},
\end{equation}
\begin{equation}
\label{tennisequ}
\H_{1,\psi}(u;z,w)=(z-uw)(z-u^{-1}w)(z^2+w^2-z^2w^2)\frac{T}{1-T}.
\end{equation}
In particular, the special cases \eqref{conj} and \eqref{tenniseq},
obtained by setting $u=1$, also hold.
\end{theorem}
\begin{remark}
  The authors have recently learned of an independent proof equation
  \eqref{conju} due to Rains and Warnaar. Also, these type of
  identities are related to work of Nekrasov and
  Okounkov~\cite{NekOk1} and are discussed in the physics literature;
  see for example~\cite{Iqbal-et-al} and~\cite{awata-kanno}[(1.4)].
\end{remark}
\begin{corollary}
Conjecture~4.2 of~\cite{HaLeVi2} is true.
\end{corollary}
\begin{proof}
This is just a reformulation of \eqref{tenniseq} in view of
\eqref{conj}
\end{proof}

The conjecture Conjecture~\ref{genconj} is purely topological in the
sense that the underlying Riemann surface of genus $g$ enters only via
its fundamental group. When incorporating the structure of this
surface as a smooth projective algebraic curve $X$ we need the data
coming from the action of Fobenius on its $l$-adic cohomology.  The
variables $u_i$ above correspond to the eigenvalues of the Frobenius
automorphism acting on the first cohomology group of $X$.  More
precisely, let $X$ be a smooth projective curve over $\F_q$ and let
$\alpha_1,\cdots,\alpha_g,q/\alpha_1,\cdots,q/\alpha_g$ be the
eigenvalues of Frobenius on $H_{\text{\'et}}^1(X,\Q_l)$. The following
extension of Conjecture~\ref{genconj} was essentially formulated by
Mozgovoy~\cite{Mozgovoy}[Conj. 3.2] for $k=0$.
\begin{conjecture}
\label{genconjfull}
With the above notation we have
\label{uconj}
\begin{enumerate}
\item The coefficients $\mathbb{H}_{\bm{\lambda}}(u;-z,w)$ are 
polynomials in $u,z,w$ with non-negative integer coefficients. 
\item
 The number of stable Higgs
  bundles on $X$ of rank $n$ and degree coprime to $n$ is given by
  $q^{2(1+(g-1)n^2)} \mathbb{H}_{(n)}(\alpha/\sqrt q;1,\sqrt q)$,
  where $\alpha:= (\alpha_1,\cdots,\alpha_g)$.
\end{enumerate}
\end{conjecture}

We have tested part 1 of this conjecture with the help of MAPLE for
the the ranges of values shown in Table \ref{maintab}.  The $g=0$ case
of course does not involve the new parameters $u_i$, and so the first
entries of this table are simply confirming the original conjectures
of~\cite{HaLeVi1} for these values. The columns which include all
values of $n$ follow from the Cauchy product formula for modified
Macdonald polynomials, and Theorem~\ref{main-ident}.

\begin{table}
\caption{Tested ranges for Conjecture \ref{uconj}}
\label{maintab}
\begin{center}
\begin{tabular}{|c||c|c|c|c|c|c|c|c|c|c|} 
  \hline 
  $g$ & $0$ &0 & 0& $1$ & 1&1 &1 & $2$ &$2$ &$2$\\
  \hline 
  $k$ & $\leq 2$ & 3 & $\leq 7$ & 0 & 1 & 2 & 3 & 0 & 1 & 2 \\
  \hline 
  $n$ & $<\infty$ & $\leq 7$ & $\leq 3$ & $<\infty$ & $\leq 6$ & $\leq
                                                                 5$ &
                                                                      $\leq
                                                                      4$&
                                                                          $\leq
                                                                          6$
                                       & $\leq 5$ & $\leq 4$ \\    
  \hline 
\end{tabular}
\end{center}
\end{table}

Consider part 1 of Conjecture~\ref{genconjfull} for $k=0$ and
$n\leq 2$.  Let $P$ be be the polynomial
$$
P(z,w):=\prod_{i=1}^g(z-u_iw)(z-u_i^{-1}w).
$$
Note that $\H_{(1)}=P$ satisfies  the claim.

It is straightforward to verify that
\begin{equation}
\label{H_2-fmla}
\frac{\H_{(2)}(u;z,w)}{P(z,w)}= \frac{P(z,w^3)}{(z^2-w^2)(1-w^4)}
+\frac{P(z^3,w)}{(z^2-w^2)(z^4-1)}
-\tfrac12 \frac{P(z,w)}{(z^2-1)(1-w^2)}
-\tfrac12 \frac{P(z,-w)}{(z^2+1)(1+w^2)}
\end{equation}
\begin{proposition}
The rational function $\H_{(2)}(u;z,w)/P(z,w)$ is a polynomial in $z,w$.
\end{proposition}
\begin{proof}
  Note that $P$ is a generic polynomial of $z,w$ with the property
  that it is fixed by $(z,w)\mapsto (w,z)$ and
  $(z,w)\mapsto(-z,-w)$. So to prove that $\H_{(2)}$ is a polynomial
  it is enough to check that the expression, say $C_{m,n}(z,w)$, on
  the right hand side of~\eqref{H_2-fmla} for
  $P(z,w):=z^mw^n+z^nw^m$, where $m,n$ run over all pairs of
  non-negative integers with $m+n$ even, is a polynomial.

 To verify this we can form the generating series, say for the case of
 $m,n$ both even,
$$
C(z,w;x,y):=\sum_{r,s\geq 0} C_{2r,2s}(z,w)\ x^ry^s.
$$
The result is the rational function
\begin{eqnarray*}
 C(z,w;x,y):=&
        \frac1{(z^2-w^2)(1-w^4)(1-z^2x)(1-w^6y)}
        +\frac1{(z^2-w^2)(1-w^4)(1-z^2y)(1-w^6x)}\\
       & +\frac1{(z^2-w^2)(z^4-1)(1-z^6x)(1-w^2y)}
        +\frac1{(z^2-w^2)(z^4-1)(1-z^6y)(1-w^2x)}\\
       & -\tfrac12\frac1{(z^2-1)(1-w^2)(1-z^2x)(1-w^2y)}
        -\tfrac12\frac1{(z^2-1)(1-w^2)(1-z^2y)(1-w^2x)}\\
        &-\tfrac12\frac1{(z^2+1)(1+w^2)(1-z^2x)(1-w^2y)}
        -\tfrac12\frac1{(z^2+1)(1+w^2)(1-z^2y)(1-w^2x)}
\end{eqnarray*}
Now we can put this expression in a computer and verify that the
denominator of $C$ is
$$
 (1-z^2x)  (1-z^2y)  
(1-w^2x)  (1-w^2y) 
(1-z^6x)  (1-z^6y)
(1-w^6x) (1-w^6y)  
$$
and this implies that the power series coefficients $C_{2r,2s}(z,w)$
are polynomial.  A similar argument proves the analogous statement for
the case where $m,n$ are both odd.
\end{proof}

\begin{remark}
To verify that the coefficients of $C_{m,n}(z,w)$ are non-negative
appears to be quite tricky.
\end{remark}

We can now check part 2 of Conjecture~\ref{genconjfull} by
specializing $\H_{(2)}(u;z,w)$ at $z=1$ and comparing the result to
the calculation of the number of points over $\F_q$ of the Higgs
moduli space by Schiffmann~\cite{schiffmann}. After some tedious
manipulations we find that $\H_{(2)}(u;1,w)=Q(w)A(w)$ where
$Q(w):=P(1,w)$ and
\begin{equation}
\label{H_2-fmla}
A(w):= \frac{Q(w^3)}{(1-w^2)(1-w^4)}
+\tfrac14\frac{(w^2-3)Q(w)}{(1-w^2)^2}
+\tfrac12 \frac{2gQ(w)-w\partial Q/\partial w}{1-w^2}
-\tfrac14 \frac{Q(-w)}{1+w^2}.
\end{equation}

\begin{proposition}
Part 2 of Conjecture~\ref{genconjfull} holds for $n\leq 2$.
\end{proposition}
\begin{proof}
  As mentioned, we compare the specialization
  $\H_{(n)}(\alpha/\sqrt q;1,\sqrt q)$ to the formula obtained by Schiffmann
  in~\cite{schiffmann} (specifically the examples given right after
  Theorem~1.6). This is immediate for $n=1$ and routine for $n=2$
  using~\eqref{H_2-fmla}.
\end{proof}
Our first result, Theorem \ref{gluethm}, is an inductive formula for
the power series $\Omega(u;z,w)$ associated to a Riemann surface
obtained by gluing along the punctures of two other surfaces.  This
reduces the computation of $\Omega(u;z,w)$ to the case of genus zero
with three punctures, as often happens in TQFT-type formulas.  It
would be very desirable to establish the gluing relations at $u=1$ on
the character variety side.  As an application we obtain a proof of
Theorem~\eqref{main-ident}.  

Our main tool turns out to be a \emph{vertex operator} discovered by
Nekrasov, Okounkov, and the first author in~\cite{CaNekOk1}, described
in Theorem \ref{cnothm} below.  The result states
\begin{equation}
\label{CNOeq}
\Gamma(u)\Ht_\lambda= 
\sum_{\mu} \frac{N_{\lambda,\mu}(u;q,t)}{N_{\mu,\mu}(1;q,t)} \Ht_\mu,
\end{equation}
where $\Gamma(u)$ is an composition of a homomorphism 
with a multiplication map with an explicit description in
the power sum basis, and
$N_{\lambda,\mu}(u)$ is an explicit product whose diagonal entries satisfy
\[\frac{N_{\lambda,\lambda}(uz^{-1}w^{-1};z^2,w^2)}
{N_{\lambda,\lambda}(z^{-1}w^{-1};z^2,w^2)} 
=\mathcal{H}_{1,\lambda}(u;z,w).\]
This operator has a natural geometric interpretation in terms of
the torus-equivariant $K$-theory of $\Hilb_n \C^2$, the 
Hilbert scheme of points in the complex plane. The connection with
modified Macdonald polynomials is due to Haiman's 
study of the isomorphism of 
Bridgeland King and Reid in this case, viewing the Hilbert scheme
as a resolution of singularities~\cite{BKR,HaiPol,Haivan}.
This operator extends to $K$-theory
some previous work due to Okounkov and the first author,
who introduced a family of geometrically defined operators
on the cohomology 
\[\bigoplus_{n\geq 0} H^*(\Hilb_n S,\mathbb{Q}),\]
for a general smooth quasiprojective surface $S$,
and proved an explicit formula in terms of Nakajima's Heisenberg
operators~\cite{CaOk,Nak2}. 
The original operator of~\cite{CaOk} is recovered in the case when
the surface $S=\C^2$ by essentially 
taking the Jack polynomial limit of the Macdonald polynomials.
See~\cite{CaNekOk1} for details.

\paragraph{Acknowledgements.} The authors would like to thank
A.~Mellit, T.~Hausel and O.~Schiffmann for useful conversations.
Erik Carlsson was supported by the International Center for
Theoretical Physics, as well as the Center for Mathematical 
Sciences and Applications at Harvard University while working
on this paper.

\section{Preliminaries on partitions and symmetric functions}

First, we recall some notations on partitions and symmetric 
functions, which can be found in Macdonald's book~\cite{Mac}. 
For more on the plehystic operations, we also recommend 
~\cite{Ber,Hai1}. 

\subsection{Partitions}

A partition $\mu$ is a nonstrictly decreasing 
sequence of nonnegative integers eventually stablizing to 
zero. Its length $l(\mu)$ is the number of nonzero terms,
and the norm $|\mu|$ is the sum of the entries. 
for each $i\geq 1$, we let $m_i(\mu)$ denote the number of 
times $i$ appears in $\mu$. 
For any partition, we have the corresponding 
\emph{Young diagram} consisting of its boxes, namely the 
the ordered pairs $s=(i,j)$ such that $1\leq j\leq\mu_i$. 
For any ordered pair $(i,j)$, we define the arm,
leg, and hook length in $\mu$ by the formula 
\[a_\mu(s)=\mu_i-j,\quad 
l_\mu(s)=\mu'_j-i,\quad 
h_\mu(s)=a(s)+l(s)+1\]
where $\mu'$ is the conjugate partition to $\mu$, i.e. 
the partition whose columns are the rows of $\mu$ in the 
Young diagram. Notice that the arm and leg length make 
sense even if the ordered pair $s$ is not a box of $\mu$,
in which case they may take negative values. 
Another important statistic we will need is 
\[n(\mu)=\sum_{i} (i-1)\mu_i.\]

For instance, if $\mu=[5,3,3]$, then the Young diagram 
would be 
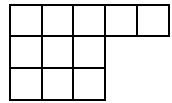
\[\yng(5,3,3)\]
The conjugate partition would be $\mu'=[3,3,3,1,1]$,
and we would have 
\[l(\mu)=3,\quad |\mu|=|\mu'|=9,\quad n(\mu)=9.\]
The arm, leg, and hook length of the box $s=(1,2)$
are given by 
\[a_\mu(s)=3,\quad l_\mu(s)=2,\quad h_\mu(s)=6.\]

An $r$-core is a partition such that none of its hook numbers
are divisible by $r$. For instance, the partition $\lambda=(5,3,1,1)$
is a $3$-core, because its hook numbers are
\[\{1,1,1,2,2,2,4,5,5,8\},\] none of which are divisible by $3$.  We
now recall the bijection between $r$-cores and lattice vectors
of~\cite{gks}. Given an $r$-core
$\lambda=(\lambda_1,\lambda_2,\ldots)$ let
$n=(n_0,\ldots,n_{r-1}),\in \Z^r$ be the vector with coordinates
$$
n_k:= \left\lfloor\frac{\lambda_i-i}t\right\rfloor +1, 
\qquad i = \min\{\nu \,|\,\lambda_\nu-\nu\equiv k\bmod~r\}. 
$$
Set $v:=rn+\rho$, where 
$\rho:=(-k,\ldots,k)\in \Z^r$ and $r=2k+1$. This 
yields a one-to-one correspondence between $r$-cores and integer 
vectors $v=(v_{-k},\ldots,v_k)$, such that 
$$
v_i\equiv i \bmod r, \quad -k\leq i \leq k, \qquad v_{-k}+\cdots +v_k=0. 
$$
Under this correspondence 
$$
|\lambda| = \frac 1{2t}(v_{-k}^2+\cdots+v_k^2)-\frac{r^2-1}{24}
$$
and if $v'$ is the vector corresponding to the dual partition 
$\lambda'$ of $\lambda$ we have 
$$
v'_\alpha=-v_{-\alpha},
$$
where the indices are read modulo $r$. 

For example, if $\lambda=(5,3,1,1)$ is viewed as a $3$-core, then $n=(0,2,-2)$ and 
hence $v=(-1,6,-5)$. We check that indeed 
$10=|\lambda|=((-1)^2+6^2+(-5)^2)/6-(3^2-1)/24$. 

\begin{remark}
  The vector $v$ associated to an $r$-core coincides with the 
  $V$-coding of~\cite{Han} up to re-ordering (his indices are in the 
  range $0\leq i \leq r-1$ whereas we are using $-k\leq i \leq k$). 
  The vector $n$ is also the weight vector in~\cite[\S 4]{Las}
  associated to $\lambda$. 
\end{remark}

\subsection{Symmetric functions}
Let $\Lambda_R=\Lambda_R[X]$ 
denote the ring of symmetric functions 
in an infinite set of variables $X$ over a ring $R$,
and let $p_\mu,m_\mu,h_\mu,e_\mu,s_\mu\in\Lambda=\Lambda_{\Q}$ 
denote the power sum, monomial, 
complete, elementary, and Schur symmetric polynomial 
bases indexed by $\mu$ respectively. 
We have the \emph{standard inner product} on $\Lambda$
described in a few different ways by 
\begin{equation}
\label{hall}
(p_\mu,p_\nu)=\delta_{\mu,\nu}z_\mu,\quad 
(h_\mu,m_\nu)=\delta_{\mu,\nu},\quad 
(s_\mu,s_\nu)=\delta_{\mu,\nu},
\end{equation}
where 
\[
z_\mu=\prod_{i\geq 1} i^{m_i} m_i!, \qquad m_i:=m_i(\mu).
\]

A homomorphism of $\Lambda_R$ into 
another $R$-algebra may be 
defined by specifying its value on the algebra generators 
$p_i$ for $i$ a nonnegative integer. 
A particularly useful example of such a homomorphism is 
that of \emph{plethystic substitution}. 
Let $(R,\lambda)$ be a $\lambda$-ring, and let 
$A\mapsto A^{(k)}$ denote the corresponding Adams 
homomorphisms. For instance, if $A$ is an element 
of a ring of rational functions, polynomials,
or Laurent polynomials in some sets of variables,
we will let $A^{(k)}$ denote the element obtained from 
$A$ by substituting $x=x^k$ for each variable $x$. 
Another example is the ring $\Lambda_{R}$. In this case, 
we extend the operations on $R$ by setting $p_{i}^{(k)}=p_{ik}$
for the power sum generators, which is compatible with specialization 
to finitely many variables. 
For instance, we would have 
\[A=\frac{1-q}{1-t} \Rightarrow 
A^{(k)}=\frac{1-q^k}{1-t^k}.\]

We use the usual language that
\[\lambda(TA) = \sum_{n\geq 0} \lambda^n(A) (-T)^n\in R[[T]]\]
for some formal variable $T$.
Now consider the \emph{plethystic exponential}
\begin{equation}
\label{pexp}
\Exp\left(TA\right) :=
\lambda\left(-TA\right)=
\exp\left(\sum_{k\geq 1} 
\frac{A^{(k)}T^k}{k}\right).
\end{equation}
For instance, if $R=\Lambda[X]$, we would have that
\[\Exp(TX)=\Exp(Tp_1)=\sum_{n\geq 0} T^nh_n\in \Lambda[X][[T]].\]
It has an inverse called the \emph{plethystic logarithm}
defined as follows: if
\[F=1+\sum_{n\geq 1} A_n T^n,\]
then we let
\[\Log(F)=\sum_{n\geq 1} V_nT^n,\]
where we define coefficients $U_n,V_n$ by
\[V_n=\frac{1}{n} \sum_{d|n} \mu(d) U_{n/d}^{(d)},\quad
\log(F)=:\sum_{n\geq 1} U_n \frac{T^n}{n}.\]
See~\cite{hausel-villegas} for details.

In the case of polynomials or Laurent polynomials,
we will let $\lambda(A)$ denote the evaluation at $T=1$
of the rational function with power series 
\[\lambda(TA)=\sum_{i\geq 0} (-T)^i \lambda^i(A)\in R[[T]],\]
provided that it exists. 
A simple example is 
\[\lambda\left(q-2qt^{-2}\right)=\frac{1-q}{(1-qt^{-2})^2}.\]
In the case of symmetric functions, if $f\in \Lambda$ has
no constant term, we may also refer to its exponential
$\Exp(f)$, leaving off the formal variable $T$. In this case,
we have not defined an honest element of $\Lambda$, but
it does make sense to consider a pairing such as
$(\Exp(f),g)$, because only finitely many terms in 
the sum \eqref{pexp} contribute. For that reason, it 
only really defines an element of the dual space.

For any element $A\in R$, we have a homomorphism 
\[\Lambda\rightarrow R,\quad 
p_k \mapsto A^{(k)}.\]
We denote the image of $f\in \Lambda$ by $f[A]$. 
The reason for this notation is that 
\[f[x_1+\cdots+x_n]=f(x_1,\ldots,x_n),\]
the evaluation of $f$ on some finite set of variables 
$x_i$. Under this map, we have $e_i[A]=\lambda^i(A)$. 
It is called a plethystic homomorphism because if 
and $f,g\in \Lambda$ are the characters of 
representations of the general linear group 
corresponding to some polynomial functors $F,G$
of vector spaces,
then $f[g]$ is the character of the composition $F\circ G$. 

A second homomorphism is the following operator,
which will play a central role in this paper: if $A\in R$, then 
\begin{equation}
\label{Gpdef}
\Gp^{m}(A) : \Lambda_R \rightarrow \Lambda_R,\quad
f \mapsto f[X+mA],
\end{equation}
where are using $X$ in place of $p_1=p_1[X]$. 
The variable $m$ may either be taken as a number,
or may be viewed as a symbol that is not affected by the plethystic 
operation, i.e. $m^{(k)}=m$. 
It is not hard to verify that
\begin{equation}
\label{Gpexp}
\Gp^m(A)=\exp\left(m\sum_{i>1} A^{(i)}\partial_i\right),
\end{equation}
where $\partial_k$ is the operator differentiation by $p_k$
on $\Lambda_R=R[p_1,p_2,\ldots]$. 
Combining equations \eqref{Gpexp} and \eqref{pexp},
we can see that its dual operator
under the standard inner product is given by a multiplication operator:
\begin{equation}
\label{Gmdual}
\left(\Gp^m(A)f,g)\right)=\left(f,\Gm^m(A)g\right),\quad
\Gm^m(A)f= \Exp\left(mAX\right)f.
\end{equation}
We will also write 
\[\Gamma_{\pm}(A)=\Gamma_{\pm}^1(A),\quad \Gamma_{\pm}=\Gamma_{\pm}(1)\]
so that $\Gamma_{\pm}^m(A)=\Gamma_{\pm}(A)^m$ for integers $m$. 
For this paper, an expression of the form $\Gm(A)\Gp(B)$ will be
called a \emph{vertex operator}.
It follows easily from the definitions that 
\begin{equation}
\label{vocr}
\Gp(A)\Gm(B)=
\Exp\left(AB\right)\Gm(B)\Gp(A) 
\end{equation}
provided both sides are convergent as power series. 
A useful example is 
\begin{equation}
\label{vocr0}
\Gp(x^{-1})\Gm(y)=
(1-yx^{-1})^{-1}\Gm(y)\Gp(x^{-1})\in \End(\Lambda)[[x^{\pm 1}]][[y]]. 
\end{equation}

A third plethystic homomorphism is 
\begin{equation}
\label{ups}
\Upsilon_{A}f=f[AX]. 
\end{equation}
For instance, we have
\[x^d=\Upsilon_{x} : p_k \mapsto x^k p_k,\]
where $d$ is the operator that multiplies a homogeneous 
polynomial $f$ of degree $k$ by $k$. 
A second common example is the usual automorphism 
\[\omega = (-1)^d \Upsilon_{-1} : p_k \mapsto (-1)^{k-1}p_k\]
It follows from definitions that we have the commutation relations 
\begin{equation}
\label{upscr}
\Gp(B)\Upsilon_A=\Upsilon_A\Gp(AB),\quad 
\Upsilon_A \Gm(B)=\Gm\left(AB\right)\Upsilon_A,\quad 
\Upsilon_A \Upsilon_B = \Upsilon_{AB}. 
\end{equation}

Let $\Lambda_{q,t}=\Lambda_{\Q(q,t)}$, with the Macdonald 
inner product 
\begin{equation}
\label{macsppleth}
(f,g)_{q,t}=\left(f,\Upsilon_Ag\right),\quad A=\frac{1-q}{1-t}
\end{equation}
or in the traditional notation,
\begin{equation}
\label{macsp}
(p_\mu,p_\nu)_{q,t}=\delta_{\mu,\nu}
z_{\mu} \prod_i \frac{1-q^i}{1-t^i}. 
\end{equation}
Let $P_\mu=P_\mu[X;q,t] \in \Lambda_{q,t}$ denote the usual Macdonald polynomials,
and let $J_\mu$ be the integral form defined by 
\begin{equation}
\label{Jmu}
J_\mu[X;q,t]=\prod_{s\in \mu} (1-q^{a(s)}t^{l(s)+1})P_\mu[X;q,t],
\end{equation}
whose expansion in the monomial basis has coefficients which 
are integer valued polynomials in $q,t$. 

The \emph{modified Macdonald inner product}~\cite{Ber},
is a variation of \eqref{macsp}
which is symmetric in $q,t$ defined by 
\begin{equation}
\label{hmsppleth}
(f,g)_*=(f,\Upsilon_{-M}
g)=(\Upsilon_{1-t}f,\Upsilon_{t-1}g)_{q,t},\quad M=(1-q)(1-t),
\end{equation}
or, in the power sum basis by 
\begin{equation}
\label{hmsp}
(p_\mu,p_\nu)_*=\delta_{\mu,\nu} (-1)^{l(\mu)}z_\mu 
\prod_i (1-q^{\mu_i})(1-t^{\mu_i}). 
\end{equation}
This differs from the usual notation~\cite{Ber} by a sign of $(-1)^d$,
but is more natural for the purposes of our paper. 
The \emph{modified Macdonald polynomials} are defined by 
\begin{equation}
\label{Hmu}
\tilde{H}_\mu=t^{-n(\mu)}\Upsilon_{(1-t^{-1})^{-1}} J_\mu[X;q,t^{-1}]
\end{equation}
and it follows easily that they are orthogonal with respect to 
the inner product \eqref{hmsp}. Their norms are given by 
\begin{equation}
\label{hnorm}
\left(\tilde{H}_\mu,\tilde{H}_\nu\right)_*=
\delta_{\mu,\nu}(-1)^{|\mu|}\prod_{s\in \mu} 
(q^{a_\mu(s)}-t^{l_\mu(s)+1}) 
(t^{l_\mu(s)}-q^{a_\mu(s)+1}). 
\end{equation}
Unlike the 
usual Macdonald polynomials, these polynomials 
have the following symmetry property in $q,t$:
\[\Ht_{\mu} = \Ht_{\mu'} \big|_{q=t,t=q}\]
As explained in~\cite{Ber},
they are eigenfunctions of the modified 
Macdonald operator, i.e. 
\begin{equation}
\label{mmo}
D_0 \Ht_{\mu}=\psi_{\mu}(q,t)\Ht_\mu,
\end{equation}
where 
\[\psi_\mu(q,t)=1+MB_\mu(q,t),\quad 
B_\mu(q,t)=\sum_{(i,j)\in \mu}q^{j-1}t^{i-1},\] and 
\[\sum_{x \in \Z} x^k D_k := \Gm(-x)\Gp(Mx^{-1}) \in 
\End(\Lambda_{q,t})[[x^{\pm 1}]].\]

The modified Macdonald polynomials arose in Garsia, Haiman, 
and Procesi's work on the $n!$ conjecture~\cite{GarPro1,HaiPol},
in which they arise as the Frobenius character of the Garsia-Haiman
module. In the course of the proof, 
they were given the following geometric interpretation:
let $\Hilb_n \C^2$ denote the Hilbert scheme of points on $\C^2$,
which is a smooth complex algebraic variety parametrizing
ideals $I\in \C[x,y]$ such that $\dim_{\C}(\C[x,y]/I)=n$.
There is an action of a two-dimensional torus $T$ on this space,
by pulling back ideals from the action
\[T \curvearrowright \C^2,\quad 
(q,t)\cdot (x,y)=(q^{-1}x,t^{-1}y).\]
The fixed points correspond to Young
diagrams of norm $n$, paramerizing ideals generated by monomials:
\[\left(\Hilb_n \C^2\right)^T=\left\{I_\mu:|\mu|=n\right\},\quad
I_\mu=(x^{\mu_1},x^{\mu_2-1}y,\ldots,y^{\mu_l}) \subset S=\C[x,y].\]
There is a bundle $\mathcal{P}$ on $\Hilb_n \C^2$ 
of rank $n!$ called the
\emph{Procesi bundle}, that is equivariant under the action
of $T$, equipped with a fiberwise action of the symmetric
group $S_n$ in such a way that every fiber is isomorphic to
the regular representation.
The fiber $\mathcal{P}_\mu$ over any fixed point 
$I_\mu \in \Hilb^T_n \C^2$ is therefore a representation 
of $S_n \times T$, that turns out to be 
isomorphic to the Garsia-Haiman module $R_\mu$,
and the $n!$ conjecture amounts to the statement that
$\dim_\C R_\mu=n!$.

The modified
Macdonald polynomial $\Ht_\mu$ arises as the Frobenius
characteristic of this module, meaning that if
\[\ch R_\mu = \sum_{\lambda} a_{\lambda,\mu}(q,t) \chi_\lambda\]
describes the decomposition into irreducibles over $S_n$, then
\[\Ht_\mu=\sum_{\lambda} a_{\lambda,\mu}(q,t) s_\lambda\]
is the description of $\Ht_\mu$ in the Schur basis. For instance,
for $n=2$ we would have
\[\Ht_{[2]}=s_{[2]}+qs_{[1,1]},\quad
\Ht_{[1,1]}=s_{[2]}+ts_{[1,1]},\]
showing that both modules have one copy of the trivial representaiton
and one copy of the alternating representation, but with different
torus actions on each component.

\section{The Vertex operator}

Several useful quantities have translates on either side. 
First, consider the equivariant Euler characteristic of a pair of 
monomial ideals 
\[\chi_{\mu,\nu}= \sum_{i\geq 0} (-1)^i \ch \Ext_S^i(I_\mu,I_\nu) 
\in \Z((q,t)),\]
where the $\ch$ refers to the torus character. The Euler 
characteristic is a well-behaved quantity that may be computed in 
any resolution of either ideal, which leads to the following nice 
expression:
\begin{equation}
\label{chiform}
\chi_{\mu,\nu}=M I_\nu\overline{I}_\mu,\quad \overline{f(q,t)}=f(q^{-1},t^{-1}),
\end{equation}
where we are using $I_\mu$ to denote both the ideal and 
its torus character $I_\mu=M^{-1}-B_\mu$,
noticing that $M^{-1}$ is the character of $S$. 

In~\cite{CaOk}, the authors introduced a family of 
classes in the (equivariant)
$K$-theory of $\Hilb_m S \times \Hilb_n S$ for any smooth 
quasiprojective surface $S$ together with a line bundle on it, 
and proved that its 
Euler characteristic defines
a \emph{vertex operator} in Nakajima's Heisenberg operators.
In the case when $S=\C^2$ with the 
above torus action and trivial bundle, it is the class of
an honest an equivariant bundle $E$ of rank $m+n$
on $\Hilb_m \C^2 \times \Hilb_n \C^2$, whose fibers map 
be described in terms of $\Ext$-groups after extending ideals
to sheaves on $\mathbb{CP}^2$. The torus characters of the 
fibers of this bundle over a fixed point 
$(I_\mu,I_\nu) \in \Hilb_m \C^2 \times \Hilb_n \C^2$
are given by
\[E_{\mu,\nu}=E_{\mu,\nu}(q,t)=\overline{\chi}_{\phi,\phi}-
\overline{\chi}_{\mu,\nu}
\in \Z[q^{\pm 1},t^{\pm 1}].\]
It is well known that $E_{\mu,\mu}$ is the character of the cotangent 
bundle to the Hilbert scheme at the torus fixed point $I_\mu$. 
The authors proved the combinatorial formula, generalizing the usual
formula for $E_{\mu,\mu}$:
\begin{equation}
\label{Ecomb}
E_{\mu,\nu}=\sum_{s \in \mu} q^{-a_\nu(s)}t^{l_\mu(s)+1}+
\sum_{s\in\nu} q^{a_\mu(s)+1}t^{-l_{\nu}(s)}. 
\end{equation}

Now for any pair of partitions, let us define 
\[N_{\la,\mu}=N_{\la,\mu}(u;q,t)=(-u)^{-|\mu|}q^{n(\mu')} t^{n(\mu)}
\lambda\left(uE_{\la,\mu}\right).\]
If only one partition is specified, we will assume that $\nu=\mu$. 
If no variable is specified, we will assume that $u=1$. 
The modified Macdonald inner product 
can now be written as 
\[(\tilde{H}_\la,\tilde{H}_\mu)_*=
\delta_{\la,\mu}N_\la.\]
On the other hand, we also have that
\[\mathcal{H}_{1,\lambda}(u;z,w)=
\frac{N_{\lambda,\la}(uz^{-1}w^{-1};z^2,w^2)}
{N_{\la,\la}(z^{-1}w^{-1};z^2,w^2)}\]

Now define the following vertex operator introduced in~\cite{CaNekOk1}:
\begin{equation}
\label{VOdef}
\Gamma(u):=\Gm\left(\frac{u^{-1}-1}{(1-q)(1-t)}\right) 
\Gp\left(1-uqt\right) 
\end{equation}
In our notation, their main result states that 
\begin{theorem}(~\cite{CaNekOk1}) 
\label{cnothm}
We have that 
\[\left(\Gamma(u) \Ht_\lambda,\Ht_{\mu}\right)_*= 
\left(\Gp\left(1-uqt\right) \Ht_{\lambda},
\Gp\left(1-u^{-1}\right)\Ht_{\mu}\right)_*=
N_{\lambda,\mu}(u).\]
\end{theorem}
The first equality follows from the properties in 
\eqref{upscr} and \eqref{Gmdual}, as well as the definition 
of the inner product \eqref{hmsppleth}. 
The proof given uses tools such as 
the Fourier transform in symmetric functions as well as 
some formulas due to Cherednik, and will not be reproduced here. 
In fact, this operator is most naturally defined as acting
on the equivariant $K$ theory of $\Hilb_n \C^2$.
The point is that this operator extends the cohomological 
vertex operator of~\cite{CaOk} in the case where the smooth surface 
is $\C^2$, in sense that we can recover it by setting 
$q=e^{a\epsilon_1},t^{a\epsilon_2},u^{am}$ and taking the limit 
\[\lim_{a\rightarrow 0} \Upsilon_{1-q} \Gamma(u) \Upsilon^{-1}_{1-q}.\]
Under this limit, the modified Macdonald symmetric functions tend to the 
Jack symmetric functions with parameter $-\epsilon_1/\epsilon_2$. 

For instance, let us calculate the case $\lambda,\mu=[2],[1,1]$. 
We have 
\[\Ht_{[2]}= s_{[2]}+qs_{[1,1]}=\frac{1+q}{2}p_1^2+\frac{1-q}{2}p_2,\]
\[\Ht_{[1,1]}=s_{[2]}+ts_{[1,1]}=\frac{1+t}{2}p_1^2+\frac{1-t}{2}p_2.\]
We can calculate 
\[\Gp\left(1-uqt\right)\Ht_{[2]}=\]
\[\frac{1+q}{2}p_1^2+\frac{1-q}{2}p_2+(1-uqt)(1+q)p_1+(1-uqt)(1-uq^2t),\]
\[\Gp\left(1-u^{-1}\right)\Ht_{[1,1]}=\]
\[\frac{1+t}{2}p_1^2+\frac{1-t}{2}p_2-(1-u^{-1})(1+t)p_1+(1-u^{-1})(1-u^{-1}t).\]
Taking the inner product gives 
\[u^{-2}t(1-u)(1-ut)(1-uq^2t^{-1})(1-uqt)=N_{[2],[1,1]}(u).\]

\section{Macdonald identities}
\label{subsubsection-mac-id}

Before proving the main theorems, it is useful to consider the
identities in some special cases.

An interesting specialization of \eqref{conju} is the
{\it Euler specialization}: $z=\sqrt q,w=1/\sqrt q$. We get 
\begin{equation}
\label{euler-specializ}
\Log\left(\sum_\lambda  \prod_{s\in \lambda} 
\frac{(1-uq^{h(s)})(1-u^{-1}q^{h(s)})}{(1-q^{h(s)})^2}\,T^{|\lambda|}\right)=
\frac{(q-u)(q-u^{-1})}{(q-1)^2}\frac{T}{1-T}. 
\end{equation}
Setting $u=q^r$ for $r$ a positive integer this becomes 
\begin{equation}
\label{mac-id}
\sum_\lambda  \prod_{s\in \lambda} 
\frac{(1-q^{h(s)+r})(1-q^{h(s)-r})}{(1-q^{h(s)})^2}
\,T^{|\lambda|}=  \prod_{j=1}^{r-1}
\prod_{n\geq 1} (1-q^{r-j}T^n)^j 
(1-q^{-r+j}T^n)^j(1-T^n),
\end{equation}
as one easily checks that 
$$
\frac{(q-q^r)(q-q^{-r})}{(q-1)^2}=
-\sum_{j=1}^{r-1}\left(jq^{r-j}+jq^{-r+j}+1\right). 
$$
The identities~\eqref{mac-id} are a $q$-analogue of
certain identities from~~\cite{NekOk1}.  We 
give an independent proof of these identities below. 

Note that the series in \eqref{mac-id} visibly gets restricted to 
partitions that are {\it $r$-cores}; i.e., partitions with no hook of 
length $r$. 
We can give a direct proof of \eqref{mac-id} for $r=2$. Indeed, the 
$2$-cores are known to be the partitions of the form 
$\lambda=(m,m-1,\ldots,2,1)$ for some positive integer $m$. A 
calculation shows that in this case 
\begin{equation}
\label{2-core}
\prod_{s\in \lambda} 
\frac{(1-q^{h(s)+2})(1-q^{h(s)-2})}{(1-q^{h(s)})^2}
=(-1)^m\frac{q^{m+1}-q^{-m}}{q-1}
\end{equation}
and hence \eqref{mac-id} becomes Jacobi's triple product identity in 
the form 
$$
\sum_{m\geq 0} (-1)^m\sum_{k=-m}^mq^k\,T^{m(m+1)/2}= \prod_{n\geq 1}
(1-qT^n) 
(1-q^{-1}T^n)(1-T^n). 
$$

For general $r$ \eqref{mac-id} follows from the Macdonald identity for 
the affine Lie algebra $A_{r-1}^{(1)}$. Indeed, if we specialize the 
right hand side of~\cite[(0.4)]{MacAff} so that $e^\alpha=q$ for 
all simple roots $\alpha$ of $A_{r-1}$ (hence in general $e^\alpha$ is 
mapped to $q^{\height(\alpha)}$, where $\height(\alpha)$ is the height 
of $\alpha$) we obtain the right hand side of \eqref{mac-id} as there 
are $j$ roots of height $r-j$ for $j=1,2,\ldots, r-1$. To complete the 
proof of~\eqref{mac-id} we would need an analogue of 
\eqref{2-core}. This is given in Proposition~\ref{prop-t-core}, which is a 
consequence of the main result of~\cite{dehaye-han}.

The key identity linking the main series~\eqref{euler-specializ} and 
that of the Macdonald identities is the following. 
\begin{proposition}
\label{prop-t-core}
For an $r$-core $\lambda$ and associated vector $v$. 
\begin{equation}
\label{t-core}
\prod_{s\in \lambda} 
\frac{(1-q^{h(s)+r})(1-q^{h(s)-r})}{(1-q^{h(s)})^2} =
\frac{V(q^{v_{-k}},\ldots,q^{v_k})}{V(q^{-k},\ldots,q^k)},
\end{equation}
where $V$ is the Vandermonde determinant. 
\end{proposition}
\begin{proof} It follows from~\cite{dehaye-han}[Theorem 1] by taking 
  $\tau(n):=1-q^n$. 
\end{proof}
Once we have \eqref{t-core} then \eqref{mac-id} becomes the Macdonald 
identity~\cite{MacAff}[(0.4)] for the affine Lie algebra 
$A_{r-1}^{(1)}$ specialized to $e^\alpha=q$ for all simple roots.  To 
see this note that the lattice $M$ of~\cite[(0.4)]{MacAff} consists 
of the vectors $\mu=rn$ with $n=(n_0,\ldots,n_{r-1})\in \Z^r$
satisfying $n_0+n_1+\cdots +n_{r-1}=0$. The half-sum of the positive 
roots $\rho$ is $(-k,\ldots,k)$ and 
$$
\frac 1{2r}\left(\|rn+\rho\|^2-\|\rho\|^2\right)=
\frac 1{2r}(v_{-k}^2+\cdots+v_k^2)-\frac{r^2-1}{24},
$$
where $(v_k,\ldots,v_k):=rn+\rho$.  We claim that the quantity 
$\chi(\mu)$ of Macdonald's matches the right hand side of 
\eqref{t-core}. The specialization we are considering sends $e^v$ to 
$q^{\langle v,\rho\rangle}$ and hence sends the numerator of 
$\chi(\mu)$ to 
$$
\sum_{\sigma\in S_n} \sgn(\sigma) \,q^{\langle \sigma v,\rho\rangle}=
q^{-r(r-1)/2}V(q^{v_{-k}},\ldots,q^{v_k}). 
$$
Similarly, the denominator of $\chi(\mu)$ equals 
$q^{-r(r-1)/2}V(q^{-k},\ldots,q^k)$ proving our claim.

It is interesting to see what happens to the vertex operator
in this special case, i.e. when $t=q^{-1},u=q^{-r}$. 
In this case we get 
\[\Upsilon_{1-q}\Ht_\mu\big|_{t=q^{-1}}=
a_\mu(q)s_{\mu},\quad 
a_\mu(q)=q^{n(\mu)}\prod_{s \in \mu} (1-q^{h(s)}),\]
\[\Gamma'(q):=
\Upsilon_{q^{-1}-1}\Gamma(q^{-r}) \big|_{t=q^{-1}}\Upsilon_{q^{-1}-1}^{-1}=
\Gm\left(1+q+\cdots+q^{r-1}\right) 
\Gp\left(-1-q^{-1}-\cdots-q^{1-r}\right).\]
By Theorem \ref{cnothm}, we must have 
\[\left(\Gamma'(q) s_{\lambda},s_{\mu}\right)= q^{-r|\lambda|}
\frac{\prod_{s \in \lambda}(1-q^{r+a_\mu(s)+l_\lambda(s)+1}) 
\prod_{s \in \mu}(1-q^{r-a_\lambda(s)-l_\mu(s)-1})}
{a_{\lambda}(q^{-1})a_\mu(q)},\]
which can be checked to agree with the left hand side of 
\eqref{t-core} at $\lambda=\mu$. 
It should therefore be possible to show that the left hand 
side agrees with the right hand side of \eqref{t-core} in 
this case. 

We prove something more general:
\begin{proposition}
If $\la$ is an $r$-core with associated vector $\nu$, then 
\[\left(\Gm\left(x_1+\cdots+x_r\right) 
\Gp\left(-x_1^{-1}-\cdots-x_r^{-1}\right)s_\lambda,s_\lambda\right)=
\frac{\det(A^\nu)}{\det(A^\rho)}\]
where $A^{\nu}_{i,j}=x_i^{\nu_j}$. 
\end{proposition}

\begin{proof}
Using \eqref{vocr0}, we may write 
\begin{equation}
\label{berneq}
\Gm(x_1+\cdots+x_r)\Gp(-x_1^{-1}-\cdots-x_r^{-1})=
\left(\prod_{i<j} (1-x_jx_i^{-1})^{-1}\right) 
\psi(x_1)\cdots \psi(x_r) 
\end{equation}
where 
\[\psi(x)=\sum_{i\in \Z} x^i \psi_i =\Gm(x)\Gp^{-1}(x^{-1})\]
is the Bernstein vertex operator. Notice that 
\[\prod_{i<j}(1-x_jx_i)^{-1}=x_1^{r-1}x_2^{r-2}\cdots x_{r-1}\det(A^\rho)^{-1}\]

The components $\psi_i$
can be described in the Schur basis as follows: for any 
partition $\mu$, consider its Maya diagram 
\[m(\mu)=\left\{\mu_1,\mu_2-1,\mu_3-2,\ldots\right\} \subset \Z.\]
If $\mu$ is an $r$-core, then its Maya diagram has a particularly 
simple shape: it satisfies 
\[j\in m(\mu),\quad i<j,\quad i=j\mbox{ (mod $r$)} \Rightarrow i\in m(\mu).\]
If $i\in \Z-m(\mu)$, let $\tilde{\mu}$ denote the unique 
partition with 
\[m(\tilde{\mu})+1=\{i\} \cup m(\mu)\]
We have that 
\[\psi_i s_{\mu}= 
\begin{cases} \pm s_{\tilde{\mu}} & i \notin m(\tilde{\mu}) \\
0 & \mbox{otherwise}. 
\end{cases}\]
The sign is determined by the position of the insertion of $i$.

Now expand 
\[x_1^{r-1}x_2^{r-2}\cdots x_{r-1}
(\psi(x_1)\cdots\psi(x_r) s_\la,s_\la)=
\sum_{I=(i_1,\ldots,i_r)\in\Z^r} a_I x_{1}^{i_1}\cdots x_{r}^{i_r}\]
as a power series in $x_i$. Since $\lambda$ is an $r$-core,
we can see that $I$ consists of $r$ distinct entries with 
distinct remainders modulo $r$, and that $a_I=0$ for all but 
one possible underlying set of $I$. Every possible 
reordering of this set contributes, leaving us with 
precisely $\det(A^\nu)$. 
\end{proof}

Now setting $x_i=q^{i-1}$, we obtain an independent proof of
\eqref{t-core}.

\section{Proof of the main identities Theorem~\ref{main-ident}}
\label{proof-main-ident}
\begin{proof}
We start by  proving the following recursive formula:
\begin{theorem}
\label{gluethm}
We have that
\[
\Omega_{g+1,k}(u;q,t)=\varphi_{u_{g+1}} \left(\Omega_{k+2}(u;q,t)\right)
\]
where
\[
\varphi_u : \Lambda[X_{k+1}]\otimes 
\Lambda[X_{k+2}]\rightarrow \Q(q,t)
\]
is the linear map satisfying
\[\varphi_u \left(f[X_{k+1}]\otimes g[X_{k+2}]\right)=
\left(\Gamma(u) f,g\right)_*.\]
\end{theorem}
\begin{proof}
We have
\[\varphi_{u_{g+1}} \left(\Omega_{k+2}(u;q,t)\right)=
\sum_{\lambda} 
\frac{ N_\lambda(u_1)\cdots N_{\lambda}(u_g)}{N_\lambda}\tilde{H}_{\lambda}[X_1]\cdots \tilde{H}_{\lambda}[X_k]
\left(\Gamma(u_{g+1}) \Ht_\lambda,\Ht_\lambda\right)
{N_\lambda}=\]
\[\sum_{\lambda} 
\frac{ N_\lambda(u_1)\cdots N_{\lambda}(u_{g+1})}{N_\lambda}\tilde{H}_{\lambda}[X_1]\cdots \tilde{H}_{\lambda}[X_k]
=\Omega_{g+1,k}(u;q,t)
\]
using Theorem \ref{cnothm}.
\end{proof}
Using the more obvious rule that
\[\Omega_{g_1+g_2}(\{u_1,\ldots,u_{g_1+g_2}\};q,t)
[X_1,\ldots,X_{k_1-1},Y_1,\ldots,Y_{k_2-1}]=\]
\begin{equation}
\varphi \left(\Omega_{g_1,k_1}(\{u_1,\ldots,u_{g_1}\};q,t)
\otimes \Omega_{g_2,k_2}(\{u_{g_1+1},\ldots,u_{g_1+g_2}\};q,t)\right)
\end{equation}
where
\[\varphi\left(f[X_{k_1}]\otimes g[X_{k_2}]\right)=
(f,g)_*,\]
we have expressed $\Omega_{k}$ in terms of the power series
$\Omega_{0,3}$, i.e. the power series associated to 
a pair of pants in TQFT language.

By contracting from genus 0 with 2 punctures 
to genus 1 with no punctures, we have
\[\Omega_{1,0}=\Tr T^d \Gamma(u).\]
Let
\[F(x)=\Tr T^d \Gm\left(\frac{u^{-1}-1}{(1-q)(1-t)}x\right)
\Gp\left(1-uqt\right)\]
so that $F(1)=\Omega_{1,0}$.
Then we have
\[F(x)=\Tr \Gm\left(\frac{u^{-1}-1}{(1-q)(1-t)}Tx\right)
T^d \Gp\left(1-uqt\right)=\]
\[\Tr T^d \Gp\left(1-uqt\right)
\Gm\left(\frac{u^{-1}-1}{(1-q)(1-t)}Tx\right)=\]
\[\Exp\left(\frac{(u^{-1}-1)(1-uqt)}{(1-q)(1-t)}Tx\right) F(Tx),\]
and
\[F(0)=\Tr T^d \Gp(1-uqt)=\Exp\left(\frac{T}{1-T}\right)\]
Solving this recurrence, we find that
\[F(x)=\Exp\left(\frac{(u^{-1}-1)(1-uqt)}{(1-q)(1-t)}
x\frac{T}{1-T}+\frac{T}{1-T}\right).\]

Putting $x=1$ and recalling the definition of $\H(u;q,t)$ we obtain
\[
\H_1(u;q,t)=
\left(-(u^{-1}-1)(1-uqt)-(1-q)(1-t)\right)\frac{T}{1-T}=\]
\[-u^{-1}(1-uq)(1-ut)\frac{T}{1-T}.\]
This proves~\eqref{conju}.

Next, we prove the more difficult second
formula~\eqref{tennisequ}. Let us define
\[\Omega'_{k}(u;q,t)=\sum_{\lambda} 
T^{|\lambda|}
\frac{\tilde{H}_{\lambda}[X_1]\cdots \tilde{H}_{\lambda}[X_k]}
{N_\lambda} N_\lambda(u)^{g}\psi_{\mu}(q,t),\]
as well as the corresponding terms 
$H'_{k}(u;q,t)$, $\mathbb{H}'_{n,\bm{\lambda}}(u;z,w)$ in the obvious
way. 

First, we will need the following lemma.
\begin{lemma}
\label{fernlem}
We have
\[[x^0] \Exp\left(x(uqt-1)\frac{1}{1-T}+
x^{-1}(u^{-1}-1)\frac{T}{1-T}\right)=\]
\begin{equation}
\label{thetaeq}
\Exp\left((1-u^{-1})(1-uqt)\frac{T}{1-T}\right)
\end{equation}
as power series in $\C(q,t)((u))[[T]]$.
\end{lemma}
\begin{proof}
Let
\[A_N=
\Exp\left(x(uqt-1)\frac{1-T^N}{1-T}+
x^{-1}(u^{-1}-1)\frac{T-T^N}{1-T}\right)
\in \C(q,t,u,x,T),\]
and let $\mathcal{T} A_N$ denote the corresponding element
of $\C(q,t)((u,x))[[T]]$.
Then we have that
\[[x^0]\mathcal{T} A_N=
\sum_{p \in \{0,1/u,T/u,T^2/u,\ldots\}} 
\mathcal{T} \Res_{x=p} x^{-1}A_N.\]
We can check that for every $j$ we have
\[\lim_{n\rightarrow\infty} 
\ldeg_u [T^j]\Res_{x=T^n/u} x^{-1} A_N =\infty \]
by a bound that is independent of $N$, where $\ldeg_u$
is the degree of the leading term in $u$.
We can also check that
\[\Res_{x=0} x^{-1}A_N =u^N,\]
which tends to zero as $N$ becomes large.
From this we find that the left hand side of \eqref{thetaeq}
is given by
\[\lim_{N\rightarrow \infty} [x^0]\mathcal{T} A_N=
\sum_{n\geq 0} \lim_{N\rightarrow \infty} 
\mathcal{T}\Res_{x=T^n/u}x^{-1}A_N=\]
\[(1-u)\Exp\left((qt-u+1-u^{-1})\frac{T}{1-T}\right)
\sum_{n \geq 0} u^n \frac{(qtT;T)_n}{(T;T)_n}.\]
The lemma now follows by using the $q$-binomial theorem
\[\sum_{n} u^n \frac{(a;T)_n}{(T;T)_n}= 
\frac{(au;T)_\infty}{(u;T)_\infty}.\]
\end{proof}

Recall that $\psi_\mu(q,t)$ is the eigenvalue of the Macdonald
operator from equation \eqref{mmo}.  Then using an argument similar to
the proof of the first identity we have
\[\Omega'_{1,0}(u;q,t)=\Tr T^d \Gamma(u) D_0=\]
\[[x^0] \Tr T^d 
\Gm\left(\frac{u^{-1}-1}{M}\right)
\Gp\left(1-uqt\right) 
\Gm(-x)\Gp\left(Mx^{-1}\right)=\]
\[[x^0] \Exp\left(\frac{(u^{-1}-1)(1-uqt)}
{M}\frac{T}{1-T}\right)
\Exp\left(-M\frac{T}{1-T}\right)\times\]
\[\Exp\left(x(uqt-1)\frac{1}{1-T}+
x^{-1}(u^{-1}-1)\frac{T}{1-T}\right)
\Exp\left(\frac{T}{1-T}\right).\]
The result follows after applying lemma \ref{fernlem} to the part
that depends on $x$.
\end{proof}
\bibliography{vcb}

\end{document}